\documentclass[12pt]{amsart}
\usepackage{amsmath,amsthm,amsfonts,amssymb,latexsym,enumerate}
\newcommand{\SSS} {\mathcal S}
\newcommand{\TT} {\mathcal T}

\newcommand{\Aut}{{{\operatorname{Aut}}}}

\newcommand{\Irr}{{{\operatorname{Irr}}}}
\newcommand{\GL}{\operatorname{GL}}

\newcommand{\SL}{\operatorname{SL}}

\newcommand{\Cl}{\operatorname{Cl}}
\newcommand{\anz}{\operatorname{anz}}
\newcommand{\Lin}{\operatorname{Lin}}
\newcommand{\PSL}{\operatorname{PSL}}

\newtheorem{thm}{Theorem}[section]
\newtheorem{lem}[thm]{Lemma}

\newtheorem{pro}[thm]{Proposition}
\newtheorem{cor}[thm]{Corollary}
\newtheorem{que}[thm]{Question}

\newtheorem*{thmA}{Theorem A}
\newtheorem*{conA'}{Conjecture A'}
\newtheorem*{thmB}{Theorem B}

\theoremstyle{definition}

\numberwithin{equation}{section}

\begin{document}

\title[Groups with a small average number of zeros  in the character table]{Groups with a small average number of zeros \\ in the character table}

\author{Alexander Moret\'o}
\address{Departamento de Matem\'aticas, Universidad de Valencia}
\email{alexander.moreto@uv.es}

\thanks{Research  supported by Ministerio de Ciencia e Innovaci\'on PID-2019-103854GB-100, FEDER funds  and Generalitat Valenciana AICO/2020/298. This work was done while the author was visiting the University of the Basque Country. He thanks the Mathematics Department for its hospitality.}

\keywords{irreducible character,  zero of a character, supersolvable group}

\subjclass[2010]{Primary 20C15}

\date{\today}

\begin{abstract} We classify finite groups with a small average number of zeros in the character table.
\end{abstract}

\maketitle


\section{Introduction}

There has been a recent interest in the study of the probability that a random character value is zero. In \cite{mil}, it was shown that if $\chi$ is a random irreducible character of the symmetric group $S_n$ and $g$ is a random element in $S_n$, then the probability that $\chi(g)=0$ goes to $1$ when $n\to\infty$. If $q$ is a fixed prime power, the analogous statement for $\GL(n,q)$ was proved in \cite{glm}.

Even more recently, two related concepts have been introduced. On the one hand, M. Larsen and A. Miller defined in \cite{lm} the sparsity of a finite group $G$ to be the fraction of non-zero entries in the character table of $G$:
$$
\Sigma(G)=\frac{|\{(g^G,\chi)\in\Cl(G)\times\Irr(G)\mid\chi(g)\neq0\}}
{k(G)|\Irr(G)|},
$$
where $\Cl(G)$ is the set of conjugacy classes of $G$, $\Irr(G)$ is the set of irreducible characters of $G$ and $k(G)$ is the number of conjugacy classes of $G$. For our purposes, it will be more convenient to consider $\Sigma'(G)=1-\Sigma(G)$. With this notation, it was proved in \cite{lm} that if $G_n$ is a family of simple groups of Lie type with rank tending to infinity, then $\Sigma'(G_n)\to1$ when $n\to\infty$. In other words, the probability that a random entry in the character table of $G_n$ is $0$ also goes to $1$ when $n\to\infty$.

On the other hand, S. Madanha studied in \cite{mad}  the average number of zeros in a row (or, equivalently, column) of the character table of $G$:
$$
\anz(G)=\frac{|\{(g^G,\chi)\in\Cl(G)\times\Irr(G)\mid\chi(g)=0\}}
{|\Irr(G)|}.
$$
Note that $\Sigma'(G)=\anz(G)/k(G)$.

It was proved in Theorem B of \cite{mad} that if $\anz(G)<1$ then $G$ is solvable. Note that $\anz(A_5)=1$, so the hypothesis in this theorem cannot be weakened. Our first main result, which does not rely on \cite{mad}, provides a classification of these groups.

\begin{thmA}
Let $G$ be a nonabelian finite group. Then $\anz(G)<1$ if and only if $G$ is one of the following groups:
\begin{enumerate}
\item
A Frobenius group with complement of order $2$ and odd order abelian kernel.
\item
A Frobenius group with cyclic complement of order $p^n-1$ and elementary abelian kernel of order $p^n$ for some prime $p$ and positive integer $n$.
\item
The Frobenius group of order $21$.
\item
The dihedral group of order $12$ or $C_3:C_4$.
\item
An extraspecial $2$-group.
\item
The symmetric group $S_4$.
\end{enumerate}
\end{thmA}

The proof of Theorem B of \cite{mad} relies on the classification of finite simple groups by means, among other results, of a theorem on extendibility of characters of a simple group $S$ to $\Aut(S)$, for which it seems unlikely that one can find a classification-free proof (Theorem A of \cite{mad}). Our proof of Theorem A also depends on the CFSG but, as we will discuss at the end of Section 3, it is conceivable that one could, perhaps, find arguments that avoid using the classification.

Conjecture 1 of  \cite{mad} asserts that if $G$ has odd order and $\anz(G)<16/11$ then $G$ is supersolvable. (Note that the average number of zeros of the Frobenius group of order $75$ is $16/11$ so this number cannot be increased.) We prove that, in fact, there are exactly $4$ nonabelian groups with average number of zeros less than $16/11$.

\begin{thmB}
Let $G$ be a nonabelian group of odd order. Then $\anz(G)<16/11$ if and only if $G$ is the Frobenius group of order $3\cdot7, 3\cdot13, 3\cdot19$ or $5\cdot11$.
\end{thmB}

 \section{Preliminary results}

In this section, we collect several  results that will be used often.
 If $G$ is a finite group and $\SSS\subseteq\Irr(G)$ we write $\anz(\SSS)$ to denote the average number of zeros of the characters in $\SSS$. Note that $\anz(\Irr(G))=\anz(G)$. Given $\chi\in\Irr(G)$, we will write $m(\chi)$ to denote the number of conjugacy classes $g^G$ of $G$ such that $\chi(g)=0$. We will use frequently the following elementary result, often without explicit mention.

 \begin{lem}
 \label{ave}
 Let $G$ be a finite group. Suppose that $\SSS\subseteq\Irr(G)$ contains all the characters $\chi\in\Irr(G)$ with $m(\chi)<\anz(G)$. Then $\anz(\SSS)\leq\anz(G)$. In particular, if $\anz(G)<1$ and $\SSS\subseteq\Irr(G)$ contains all the linear characters of $G$, then $\anz(\SSS)\leq\anz(G)$. It follows that if $N\trianglelefteq G$ is contained in $G'$, then $\anz(G/N)\leq\anz(G)$.
 \end{lem}

 \begin{proof}
 Note that
 $$
 \anz(G)=\frac{\sum_{\chi\in\SSS}m(\chi)+\sum_{\chi\not\in\SSS}m(\chi)}
 {|\Irr(G)|}\geq
 \frac{\sum_{\chi\in\SSS}m(\chi)+|\Irr(G)-\SSS|\anz(G)}
 {|\Irr(G)|}.
 $$
 Therefore,
 $$
 \frac{|\SSS|}{|\Irr(G)|}\anz(G)\geq\frac{\sum_{\chi\in\SSS}m(\chi)}{|\Irr(G)|},
 $$
 so
 $$\anz(G)\geq\frac{\sum_{\chi\in\SSS}m(\chi)}{|\SSS|}=\anz(\SSS),$$
 as desired.
 \end{proof}

  We will use the following results on odd order groups.

  \begin{lem}
  \label{2}
  Let $G$ be an odd order group. If $\chi\in\Irr(G)$ is not linear, then $m(\chi)$ is a positive even integer.
  \end{lem}

\begin{proof}
We know that $m(\chi)>0$ by Burnside's theorem.
Suppose that $\chi(x)=0$. Since $|G|$ is odd, $x$ and $x^{-1}$ are not conjugate. Furthermore, $\chi(x^{-1})=\overline{\chi(x)}=0$. The result follows.
\end{proof}

  \begin{lem}
  \label{quo}
  Let $G$ be an odd order group. Suppose that $\anz(G)<2$. Let $N\trianglelefteq G$ such that $N\leq G'$. Then $\anz(G/N)\leq\anz(G)$.
  \end{lem}

\begin{proof}
Since $N\leq G'$, $\Irr(G/N)$ contains all the members of $\Irr(G)$ that vanish on at most  $\anz(G)$ classes, by Lemma \ref{2}. Now, the result follows from Lemma \ref{ave}
\end{proof}

Recall that a group $G$ is a Camina group if $gG'$ is a conjugacy class for every $g\in G-G'$. We will use the following result on Camina groups \cite{ds}.

\begin{thm}[Dark-Scoppola]
\label{ds}
Let $G$ be a Camina group. Then one of the following holds:
\begin{enumerate}
\item
$G$ is a $p$-group, for some prime $p$, of nilpotence class at most $3$.
\item
$G$ is a Frobenius group whose complement is  cyclic.
\item
$G$ is a Frobenius group whose complement is quaternion of order $8$.
\end{enumerate}
\end{thm}

We will also use the following characterization of Camina groups in terms of zeros of characters.

\begin{lem}
\label{char}
Let $G$ be a finite group. Then $G$ is a Camina group if and only of $\chi(x)=0$ for every nonlinear $\chi\in\Irr(G)$ and every $x\in G-G'$.
\end{lem}

\begin{proof}
 This is Lemma 3.1 of \cite{lmw}. (We remark that Lemma 3.1 of \cite{lmw} was stated for Camina $p$-groups but exactly the same argument works for arbitrary groups.)
 \end{proof}

Our next result is a consequence of a well-known elementary, but often useful, lemma of T. Wolf. We refer the reader to \cite{isa} for any unexplained notation or concept.

 \begin{lem}
 \label{wolf}
 Let $G$ be a finite group and let $\varphi\in\Irr(G')$. Then there exists a unique subgroup $G'\leq U\leq G$ maximal such that $\varphi$ extends to $U$ and any irreducible character of $G$ lying over $\varphi$ vanishes on $G-U$. Furthermore, there are exactly $|U:G'|$ such characters.
\end{lem}

\begin{proof}
Put $T=I_G(\varphi)$. Note that $(T,G',\varphi)$ is a character triple with $T/G'$ abelian. By Lemma 2.2(a) of \cite{wol}, there exists a unique $U\leq T$ maximal with respect to $\varphi$ having a $T$-invariant extension to $U$. By Gallagher's theorem (Corollary 6.17 of \cite{isa}), there are $|U:G'|$ extensions of $\varphi$ to $T$. By Lemma 2.2(b) of \cite{wol}, all such extensions are fully ramified with respect to $T/U$. This implies that there are exactly $|U:G'|$ irreducible characters of $T$ lying over $\varphi$ and all of them vanish on $T-U$ (by Problem 6.3 of \cite{isa}, for instance). Now, we can use Clifford's correspondence (Theorem 6.11 of \cite{isa}) to deduce the result.
\end{proof}

This lemma will be used frequently, sometimes without further explicit mention.
We will say that $U$ is the {\it Wolf subgroup} associated to $\varphi$.

\section{Proof of Theorem A}

 Next, we start working toward a proof of Theorem A.
The next result will allow us to assume in the proof of Theorem A that $G$ has some irreducible character that vanishes on at least $3$ conjugacy classes. Given a group $G$ we set $m(G)=\max m(\chi)$, where $\chi\in\Irr(G)$. We will write $\Lin(G)$ to denote the set of linear characters of $G$.

\begin{lem}
\label{sma}
Let $G$ be a finite group. Then
\begin{enumerate}
\item
If $m(G)=1$ then $G$ is a Frobenius group with complement of order $2$ and odd order abelian kernel.
\item
If $m(G)=2$ then $G$ is an extension of a group of order $2$ by a Frobenius group with complement of order $2$ and odd order abelian kernel, a Frobenius group with complement of order $3$ and abelian kernel, $S_4$ or $\anz(G)\geq1$.
\end{enumerate}
In particular, if $|G:G'|\leq2$ then $\anz(G)<1$ if and only if $G=S_4$ or $G$ is a Frobenius group with complement of order $2$ and odd order abelian kernel.
\item
\end{lem}

\begin{proof}
Part (i) is Proposition 2.7 of \cite{chi}. Part (ii) follows from Theorem H of \cite{ms}. (It suffices to check that $\anz(A_5)=1$ and $\anz(\PSL_2(7))=4/3$.)

Now, we will prove the last statement. Suppose that $|G:G'|\leq2$ and  $\anz(G)<1$. By way of contradiction,  assume that there exists $\chi\in\Irr(G)$ such that $m(\chi)\geq3$. Set $\SSS=\{\chi\}\cup\Lin(G)$. Note that $|\SSS|\leq3$. Hence, using Lemma \ref{ave} we have  $\anz(G)\geq\anz(\SSS)\geq1$, a contradiction. Hence $m(G)\leq2$. If $m(G)=1$, then $G$ is one of the groups in (i).
Assume now that $m(G)=2$. It is easy to see that among the groups that appear in (ii) only $G=S_4$ satisfies that $|G:G'|\leq 2$.

Conversely, it is easy to see that if   $G=S_4$ or $G$ is a Frobenius group with complement of order $2$ and odd order abelian kernel then $\anz(G)<1$.
\end{proof}

Next, we see that Theorem A also holds when $G$ is solvable and has an irreducible character $\chi$ such that $m(\chi)=1$.

\begin{lem}
\label{qian}
Let $G$ be a solvable group. Suppose that there exists $\chi\in\Irr(G)$ such that $m(\chi)=1$. Then one of the following holds:
\begin{enumerate}
\item
$G$ is a Frobenius group with a complement of order $2$.
\item
There are normal subgroups $M,N$ of $G$ such that $M$ is a Frobenius group with kernel $N$ and $G/N$ is a Frobenius group of order $p(p-1)$ for some odd prime $p$.
\item
There are normal subgroups $M,L$ of $G$ such that $M/L=Q_8$, $G/L=\SL_2(3)$
and $M$ is a Frobenius group with kernel $L$.
\item
$G=H_pP$ where $p\in\{2,3,5,7\}$, $P$ is a normal Sylow $p$-subgroup of $G$, $P$ is either $Q_8$ or the extraspecial group of order $p^3$ and exponent $p$, $H_p$ is a cyclic subgroup of $\SL_2(p)$ of order $p^2-1$, $H_p$ acts frobeniously on $P/P'$ and trivially on $P'$, and $H_2=C_3$, $H_3=Q_8$, $H_5=\SL_2(3)$ and $H_7/H_7\cap\SL_2(7)=S_4$.
\end{enumerate}
In particular, if $G$ is solvable, $\anz(G)<1$ and there exists  some  character $\chi\in\Irr(G)$ such that $m(\chi)=1$, then $G=S_4$ or  $G$ is a Frobenius group with complement of order $2$ and abelian kernel of odd order.
\end{lem}

\begin{proof}
The first part is Theorem 1.1 of \cite{qia}. Now, we want to prove that if $G$ is one of the groups in (ii)-(iv) then either $G=S_4$ or $\anz(G)\geq1$.

Suppose now that $G$ is one of the groups in (ii). Note that if $\tau\in\Irr(G/N)$ is not linear and $\{1=x_1,\dots,x_{p-1}\}$ is a complete set of representatives of the cosets of $M$ in $G$ then $\tau$ vanishes on $x_iM$ for every $i\geq2$. It is easy to see that for every $i\geq2$,  $x_iM$ consists of at least $2$ $G$-conjugacy classes (because $x_i$ does not act frobeniously on $N$). Thus $\tau$ vanishes on at least $2(p-2)$ conjugacy classes. Therefore,
$$
1>\anz(G)\geq\anz(\{\tau\}\cup\Lin(G))\geq\frac{2(p-2)}{p}
$$
and we deduce that $p=3$. Hence $G/N=S_3$. In particular, $|G:G'|=2$ and the result follows from Lemma \ref{sma}.

Next, assume that $G$ is one of the groups in (iii). In this case
$$
1>\anz(G)\geq\anz(G/L)=1,
$$
by Lemma \ref{ave}. This is a contradiction. Finally, if $G$ is one of the groups in (iv), one can also see in the character table of those groups that $\anz(G)\geq1$, again a contradiction.
\end{proof}

 The following property of the Wolf subgroup of groups with average number of zeros less than $1$ will be essential in our proof of Theorem A.

\begin{lem}
\label{pro}
Let $G$ be a group with $\anz(G)<1$. Suppose that $G\neq D_{12}$ and $G\neq C_3:C_4$.  If $\varphi\in\Irr(G')$ then the Wolf subgroup $U$ associated to $\varphi$ is either $G'$ or $G$.
\end{lem}

\begin{proof}
By way of contradiction, assume that $G'<U<G$. Write $e=|G:U|$ and $f=|U:G'|$, so that $|G:G'|=ef$.
  Then each coset $gU$ contains at least $2$ $G$-conjugacy classes (because the size of any conjugacy class cannot exceed $|G'|$). It follows that $G-U$ contains at least $2(e-1)$ $G$-conjugacy classes.  The average number of zeros  of the $f$ irreducible characters of $G$ that lie over $\varphi$ and   the linear characters is at least
 $$
 \frac{2f(e-1)}{ef+f}=\frac{2(e-1)}{e+1}
  \geq1
  $$
  if $e\geq3$. By Lemma \ref{ave}, we conclude that $e=2$. Let $gU$ be the nontrivial coset of $U$ in $G$. If $gU$ is the union of at least $3$ $G$-conjugacy classes then, as before, the average number of zeros  of the $f$ irreducible characters of $G$ that lie over $\varphi$ and   the linear characters is at least $3f/(2f+f)=1$. Using Lemma \ref{ave} we obtain a contradiction. Hence, $gU$ is the union of exactly $2$ $G$-conjugacy classes. In particular, $G$ has some conjugacy class of size at least $|U|/2=|G|/4$. Hence $|\Lin(G)|=|G:G'|=4$. Since $e=2$, we deduce that $f=2$. We know that the $2$ irreducible characters of $G$ lying over $\varphi$ vanish on the $2$ conjugacy classes in $G-U$.

Assume first that $m(G)=2$.   By Lemma \ref{sma}, we know that $G$ is an extension of a group of order $2$ by a Frobenius group with complement of order $2$ and odd order abelian kernel. If $Z=Z(G)$ and $K/Z$ is the Frobenius complement, then we know that $gK$ is the union of two conjugacy classes for any $g\in G-K$. Any nonlinear irreducible character of $G$ vanishes on $gK$ (is induced from $K$) so the number of nonlinear irreducible characters of $G$ is at most $3$ (because $\anz(G)<1$). It is easy to conclude that $G=D_{12}$ or $G:C_3:C_4$, a contradiction.

   Hence, we may assume that there exists $\chi\in\Irr(G)$ such that $m(\chi)\geq3$. If $\chi$ is one of the two irreducible characters of $G$ that lie over $\varphi$, then the other one, say $\chi'$,  also vanishes on at least $3$ conjugacy classes. Then $\anz(\{\chi,\chi'\}\cup\Lin(G))\geq1$, a contradiction by Lemma \ref{ave}. Suppose now that  $\chi$ is not one of the two irreducible characters of $G$ that lie over $\varphi$. Let $\theta_1,\theta_2$ be these two irreducible characters. Then $\anz(\{\chi,\theta_1,\theta_2\}\cup\Lin(G))\geq1$, another contradiction by Lemma \ref{ave}. This completes the proof.
  \end{proof}

  \begin{cor}
  \label{meta}
  Let $G$ be a metabelian group with $\anz(G)<1$. Then  $G=D_{12}$, $G=C_3:C_4$ or $G$ is a Camina group.
  \end{cor}

  \begin{proof}
  Suppose that $G\neq D_{12}$ and $G\neq C_3:C_4$.
  Since $G'$ is abelian, any $\lambda\in\Irr(G')$ is linear. Hence, if $\lambda\neq1_{G'}$, $\lambda$ does not extend to $G$, so the Wolf subgroup associated to $\lambda$ is proper in $G$. By Lemma \ref{pro}, it has to be $G'$. Therefore, any nonlinear irreducible character of $G$ vanishes on $G-G'$.  By Lemma \ref{char}, $G$ is a Camina group.
  \end{proof}

  Next, we determine the metabelian Camina groups with small average number of zeros.

  \begin{lem}
  \label{pgp}
Let $G$ be a Camina $p$-group. If $\anz(G)<1$ then $G$ is an extraspecial $2$-group. Furthermore, if $p$ is odd, then $\anz(G)\geq16/11$.
\end{lem}

\begin{proof}
Let $N$ be a normal subgroup of $G$ maximal in $G'$. Therefore, $G/N$ is extraspecial (by Theorem A of \cite{fm}, for instance). Write $|G/N|=p^{2n+1}$. By the character theory of extraspecial groups, $G/N$ has $p^{2n}$ linear characters and $p-1$ nonlinear irreducible characters that vanish of $G/N-(G/N)'$. There are exactly $(p^{2n+1}-p)/p=p^{2n}-1$ conjugacy classes in this subset, whence
  $$
  \anz(G/N)=\frac{(p-1)(p^{2n}-1)}{p^{2n}+p-1}\geq16/11
  $$
  if $p\geq3$.
  In this case, using Lemma \ref{ave} and Lemma \ref{2}, we have $\anz(G)\geq\anz(G/N)\geq16/11$. Hence, we may assume that $p=2$, so $G$ is a semiextraspecial $2$-group. Write $|G:G'|=2^{2n}$ and $|G'|=2^d$. It follows from Theorem 1 of \cite{bei}  that $d\leq n$. We have that $G$ has $2^{2n}$ linear characters and $2^d-1$ nonlinear irreducible characters that vanish of $G-G'$. There are exactly $(2^{2n+d}-2^d)/2^d=2^{2n}-1$ conjugacy classes in this subset, whence
  $$
  1>\anz(G)=\frac{(2^d-1)(2^{2n}-1)}{2^{2n}+2^d-1}.
  $$
Therefore $2^{2n+d}<2^{2n+1}+2^{d+1}-2$. If $d\geq2$ then
$$
2^{2n+d}\geq2^{2n+2}=2^{2n+1}+2^{2n+1}\geq2^{2n+1}+2^{2d+1}\geq
2^{2n+1}+2^{d+1}-2,
$$
a contradiction. We conclude that $d=1$, so $G$ is an extraspecial $2$-group.
\end{proof}

\begin{lem}
\label{frob}
Let $G$ be a metabelian Frobenius group with cyclic complement. If $\anz(G)<1$ then $G$ is one of the following:
\begin{enumerate}
\item
A Frobenius group with complement of order $2$ and odd order abelian kernel.
\item
A Frobenius group with cyclic complement of order $p^n-1$ and elementary abelian kernel of order $p^n$ for some prime $p$ and positive integer $n$.
\item
The Frobenius group of order $21$.
\end{enumerate}
\end{lem}

\begin{proof}
Let $C$ be the Frobenius complement and let $K$ be the Frobenius kernel. By Lemma \ref{sma}, we may assume that $|C|=|G:G'|=n\geq3$. Any nonlinear irreducible character of $G$ vanishes on $G-K$. There are $n-1$ conjugacy classes of $G$ contained in $G-K$. Let $k$ be the number of nonlinear irreducible characters of $G$. Then
$$
1>\anz(G)\geq\frac{k(n-1)}{n+k},
$$
so $k<n/(n-2)$.

We conclude that if $n>3$, then $k=1$. Therefore, there exists a unique nontrivial orbit of $C$ on $\Irr(K)$. It follows that (ii) holds.

If $k>1$ then $n=3$ and $k=2$. Hence, (iii) holds.
\end{proof}

If $N$ is a normal subgroup of a group $G$, we write $\Irr(G|N)$ to denote the set of irreducible characters of $G$ whose kernel does not contain $N$.
Now, we are ready to complete most of the proof of Theorem A.

  \begin{thm}
  \label{first}
  Let $G$ be a finite group with $\anz(G)<1$. Then either $G=S_4$ or $G$ is metabelian.
  \end{thm}

  \begin{proof}
By Lemma \ref{sma}, we may assume that $|G:G'|\geq3$.
   We also assume that $G'$ is not abelian  and we want to reach a contradiction. In particular, since $G'$ is not abelian, $G\neq D_{12}$ and $G\neq C_3:C_4$.

   Assume first that the Wolf subgroup of every nonlinear irreducible character of $G'$ is $G$. In other words, assume that every  nonlinear irreducible character of $G'$ extends to $G$. Let $\chi\in\Irr(G|G'')$. Since $\chi$ lies over a nonlinear irreducible character $\mu$ of $G'$, $\chi$ extends $\mu$.

By way of contradiction, suppose  that $m(\chi)\geq2$.
  By Gallagher's theorem $\TT=\{\lambda\chi\mid\lambda\in\Irr(G/G')\}$ is a set of $|G:G'|$ irreducible characters. Furthermore, $m(\lambda\chi)\geq2$ for every $\lambda\in\Irr(G/G')$. This implies that $$\anz(G)\geq\anz(\TT\cup\Lin(G))\geq1,$$ a contradiction. We conclude that $m(\chi)=1$ for every $\chi\in\Irr(G|G'')$.

If $G'=G''$ then $\Irr(G|G'')$ is the set of nonlinear irreducible characters of $G$, so $m(\chi)=1$ for every nonlinear $\chi\in\Irr(G)$. By Lemma \ref{sma}(i) $G$ is solvable. This contradicts the fact that $G'=G''$. Therefore, $G''<G'$.

If we choose $\chi\in\Irr(G|G'')$ of maximal degree among the members of $\Irr(G|G'')$ then $\chi_{G''}$ is not irreducible (otherwise, we could take $\psi\in\Irr(G/G'')$ nonlinear and Gallagher's theorem implies that $\psi\chi\in\Irr(G|G'')$, contradicting the choice of $\chi$). By \cite{nav}, this implies that $\chi$ vanishes on $xG''$, for some $x\in G$. Since $m(\chi)=1$, all the elements in $xG''$ are conjugate. Now, Theorem B of \cite{gn} implies that $G''$, and hence $G$, is solvable. Now the result follows from Lemma \ref{qian}.

  Using Lemma \ref{pro},
we conclude that the Wolf subgroup of some nonlinear irreducible character of $G'$ is $G'$. Put $d=|G:G'|$. Suppose that there are $d-1$ conjugacy classes in $G-G'$. Then $G$ is a Camina group.

  By Theorem \ref{ds}, either $G$ is a $p$-group of nilpotence class at most $3$ or $G$ is  a Frobenius group with complement cyclic or quaternion. In the first case, we know  that the nilpotence class of $G$ is at most $3$ , so $[G',G']\leq\gamma_4(G)=1$, and $G$ is metabelian. This is a contradiction.

  Assume now that $G$ is a Frobenius group with cyclic complement $C$ of order $d$ and kernel $K$. Every nonlinear irreducible character of $G$ vanishes on $G-K$, and this consists of $d-1$ conjugacy classes. Recall that $d>2$. Let $k$ be the number of nonlinear irreducible characters of $G$. We have that
  $$
  1>\anz(G)\geq\frac{k(d-1)}{d+k},
  $$
  so $k<d/(d-2)$.
  Therefore, $k\leq2$.
  Note that if $k=1$ then $K$ is abelian, so $G$ is metabelian. This is a contradiction. Hence, we may assume that $k=2$, i.e., $G$ has exactly $2$ nonlinear irreducible characters. Since $G$ is not metabelian $K$ has class $2$. Hence, there exists a unique irreducible character in $\Irr(G/K'|K/K')$ and a unique irreducible character in $\Irr(G|K')$. Since $|C|=3$ and $|\Irr(G/K'|K/K')|=1$, we conclude that $|K/K'|=4$. By III.11.9 of \cite{hup}, we have that $|K'|=2$. Hence $K'$ is central in $G$ and $|\Irr(G|K')|=3$, a contradiction.

  Finally, we may assume that $G$ is a Frobenius group with quaternion complement $Q$ and kernel $K$. In this case, $G$ has $4$ linear characters. The nonlinear irreducible character of $G/K$ vanishes on $3$ conjugacy classes. If $\chi\in\Irr(G|K)$ then $\chi$ vanishes on at least $3$ conjugacy classes. The average number of zeros of the 4 linear characters and these two characters is at least $1$, so we have a contradiction.

  Hence, $G-G'$ consists of at least $d$ conjugacy classes. If $|\Irr(G|G'')|\geq2$  then $G$ has at least $2$ irreducible characters that vanish on at least $d$ conjugacy classes. The average number of zeros of these nonlinear characters and the $d$ linear characters is at least $1$, a contradiction. Thus $|\Irr(G|G'')|=1$. This implies that $G''$, and hence $G$, is solvable (for instance, using Theorem 12.5 of \cite{isa} because $G''$ has at most two character degrees). Hence, by Lemma \ref{qian}, we may assume that $m(\chi)>1$ for every nonlinear $\chi\in\Irr(G)$. Since $G$ is not metabelian, $G$ has at least $2$ nonlinear irreducible characters. We know that one of them vanishes on at least $d$ conjugacy classes and the other one vanishes on at least $2$ conjugacy classes. Hence, the average number of zeros of these two nonlinear irreducible  characters and the $d$ linear characters of $G$ is at least $1$. This contradiction proves that $G$ is metabelian.
  \end{proof}

  Now, we complete the proof of Theorem A.

  \begin{proof}[Proof of Theorem A]
  Suppose that $\anz(G)<1$. We may assume that $G\neq S_4$, $G\neq D_{12}$ and $G\neq C_3:C_4$. By Theorem \ref{first}, we know that $G$ is metabelian. Now, Corollary \ref{meta} implies that $G$ is a Camina group.
  Therefore, $G$ is one of the groups that appear in (i) or (ii) of Theorem \ref{ds}. If $G$ is a Camina $p$-group, then it follows from Lemma \ref{pgp} that $G$ is an extraspecial $2$-group. If $G$ is a Frobenius metabelian group, then it follows from Lemma \ref{frob} that $G$ is one of the groups that appear in (i)-(iii) of the statement of Theorem A.

 Conversely, it is easy to see that any of the groups in (i)-(vi) has average number of zeros less than $1$.
    \end{proof}

    Our proof of Theorem A relies on the classification of finite simple groups by means of Theorem B of \cite{gn} and Proposition 2.7 of \cite{chi}. One way to avoid using the CFSG in Theorem A would be to find classification-free proofs of the following results.

    \begin{pro}
    Let $S$ be a nonabelian simple group. Then there exists $\chi\in\Irr(G)$ such that $m(\chi)>1$.
    \end{pro}

    \begin{pro}
    Let $G$ be a finite group. Suppose that $G''<G'<G$. If there exists $x\in G-G''$ such that $xG''$ is a conjugacy class then $G$ is solvable.
    \end{pro}

    \section{Proof of Theorem B}

    In this section, we prove Theorem B.
We start with the odd order analog of Lemma \ref{pro}.
\begin{lem}
\label{pro2}
Let $G$ be an odd order group with $\anz(G)<12/5$. If $\varphi\in\Irr(G')$ then the Wolf subgroup $U$ associated to $\varphi$ is either $G'$ or $G$.
\end{lem}

\begin{proof}
By way of contradiction, assume that $G'<U<G$. Write $e=|G:U|$ and $f=|U:G'|$, so that $|G:G'|=ef$.
  Then each coset $gU$ contains at least $3$ $G$-conjugacy classes (because the size of any conjugacy class cannot exceed $|G'|$). It follows that $G-U$ contains at least $3(e-1)$ $G$-conjugacy classes.  The average number of zeros  of the $f$ irreducible characters of $G$ that lie over $\varphi$, the $f$ irreducible characters that lie over $\overline{\varphi}$ and   the linear characters is at least
 $$
 \frac{2\cdot3f(e-1)}{ef+f+f}=\frac{6(e-1)}{e+2}
  \geq12/5
  $$
   for $e\geq3$. Using Lemma \ref{ave}, we have a contradiction.
  \end{proof}

  The following corollary can be deduced exactly as Corollary \ref{meta}.

  \begin{cor}
  \label{meta2}
  Let $G$ be a metabelian odd order group with $\anz(G)<12/5$. Then $G$ is a Camina group.
  \end{cor}

  Next, we deduce the metabelian case of Theorem B.

  \begin{thm}
  \label{b}
  Let $G$ be an odd order metabelian group. If $G$ is not abelian and $\anz(G)<16/11$ then $G$ is a Frobenius group of order $3\cdot7, 3\cdot13, 3\cdot19$ or $5\cdot11$.
  \end{thm}

  \begin{proof}
  By Corollary \ref{meta2}, we know that $G$ is a Camina group. Again, we use Theorem \ref{ds}.

  If $G$ is a $p$-group, then we reach a contradiction using Lemma \ref{pgp}. Now, we may assume that $G$ is a Frobenius group with cyclic complement $C$ and kernel $K$. Write $|C|=n$. Since $|G|$ is odd, $G$ has at least two nonlinear characters. Each of them vanishes on the $n-1$ conjugacy classes that are contained in $G-K$.

  Assume first that $n=3$. Let $k$ be the number of nonlinear irreducible characters of $G$. The average number of zeros of these $k$ nonlinear irreducible characters and the $3$ linear characters is at least $2k/(k+3)\geq16/11$ if $k>6$. Hence, $G$ has $2,4$ or $6$ nonlinear irreducible characters. Therefore, $|K/\Phi(K)|=7, 13$ or $19$, respectively. In particular, $K$ is cyclic and it is easy to deduce that $\Phi(K)=1$. Hence, $G$ is a Frobenius group of order $3\cdot7, 3\cdot13$ or $3\cdot19$.

  Now, assume that $n=5$. Arguing as before, we can see that the number of nonlinear irreducible characters of $G$ is $2$. Hence, $G$ is the Frobenius group of order $5\cdot11$.

Finally, assume that $n\geq7$ and we want to reach a contradiction. The group $G$ has at least $2$ nonlinear characters that vanish on at least the $n-1$ conjugacy classes of $G$ that are contained in $G-K$. Hence, the average number of zeros of these two nonlinear irreducible characters and the $n$ linear characters is at least $2(n-1)//n+2)$. By Lemma \ref{ave}, $2(n-1)/(n+2)<16/11$ so $n<9$. Therefore, $n=7$. But then we can see that the action of $C$ on $\Irr(K)$ cannot have two nontrivial orbits. This completes the proof.
\end{proof}

Now, we are ready to complete the proof of Theorem B. We want to see that if $G$ is an odd order group with $\anz(G)<16/11$ then $G$ is one of the groups that appear in Theorem \ref{b}.

\begin{proof}[Proof of Theorem B]
Let $G$ be a minimal non-metabelian odd order (solvable) group with $\anz(G)<16/11$.  By Lemma \ref{quo}, $G''$ is a minimal normal subgroup of $G$. By Lemma \ref{quo} and Theorem \ref{b}, $G/G''$ is a Frobenius group of order $3\cdot7, 3\cdot13, 3\cdot19$ or $5\cdot11$. It follows that $G''$ is the Fitting subgroup of $G$ and the action of $G'/G''$ on $G''$ is Frobenius. Write $d=|G:G'|$.

Let $\varphi\in\Irr(G')$ be nonlinear. By Lemma \ref{pro2}, the Wolf subgroup associated to $\varphi$ is $G'$ or $G$. Suppose first that it is $G$. Therefore, $\varphi$ extends to $\chi\in\Irr(G)$ and $\chi$ is not identically zero on $G-G'$. By Lemma 3.1 of \cite{lmw}, $G$ is not a Camina group. Hence, there exists $g\in G-G'$ such that the conjugacy class of $G$ has size $|g^G|<|G'|$. Since $|G|$ is odd, $|g^G|\leq|G|/3$. The same holds for $|(g^i)^G|$ for $1\leq i<d$. Therefore, there are at least $3(d-1)$ conjugacy classes in $G-G'$. Recall that the (at least two) nonlinear irreducible characters of $G/G''$ vanish on $G-G'$. Hence, by Lemma \ref{ave}
$$
\frac{16}{11}>\anz(G)\geq\frac{6(d-1)}{d+2}
$$
and we have a contradiction since $d\geq3$.

It follows that the Wolf subgroup associated to any nonlinear irreducible character of $G'$ is $G'$. In other words, any character in $\Irr(G|G'')$ vanishes on $G-G''$.

Suppose first that $d=3$.
 Note that the number of conjugacy classes in $G-G''$ is at least $4$. The average number of  two characters in $\Irr(G|G'')$ and the linear characters of $G$ is at least
$2\cdot4/5>16/11$. This contradicts Lemma \ref{ave}. Arguing similarly, we algo get a contradiction when $d=5$. This contradiction implies that $G$ is metabelian. The result now follows from Theorem \ref{b}.

Conversely, it is routine to see that if $G$ is one of the $4$ Frobenius groups in the statement, then $\anz(G)<16/11$.
\end{proof}

\section{Open questions}

Theorem A suggests that, perhaps, the derived length of a solvable group $G$ is bounded above in terms of $\anz(G)$. Note however that this would be a strong form of the first part of Conjecture F of \cite{ms}. We propose a more modest problem.

\begin{que}
\label{fit} Is it true that there exists a real-valued function $f$ such that for every solvable group $G$ the Fitting height of $G$ is $h(G)\leq f(\anz(G))$?
\end{que}

It is definitely false that the second part of Conjecture F of \cite{ms} admits a strong version replacing $m(G)$ by $\anz(G)$: consider the Frobenius groups of order $(p^n-1)p^n$. Recall that for these groups the average number of zeros is less than $1$, while the index of the Fitting subgroup is arbitrarily large. Nevertheless, it could be true that if $\Irr_1(G)=\Irr(G)-\Lin(G)$ then $|G:F(G)|$ is bounded above in terms of $\anz(\Irr_1(G))$, which would be a strong form of Conjecture 3.2 of \cite{mor} (the analogous strong form of the second part of that conjecture could perhaps be true too). Again, we will propose a more modest goal.

\begin{que}
Is it true that there exists a real-valued function $f$ such that for every nonabelian $p$-group $G$ we have $|G|\leq f(\anz(G))$?
\end{que}

Note that if we replace $\anz(G)$ by $m(G)$ in the previous question, then the answer is affirmative by Theorem B of \cite{ms}.

We conclude with a much stronger form of Question \ref{fit}, motivated by the result in \cite{lm} mentioned in the Introduction.

\begin{que}
\label{3}
Let $\{G_n\}$ be a numerable family of solvable groups with Fitting height going to infinity when $n\to\infty$. Is it true that $\Sigma'(G_n)\to1$ when $n\to\infty$?
\end{que}

\end{document}